\documentclass{elsarticle}

\usepackage{graphicx}
\usepackage[utf8]{inputenc}
\usepackage{subfigure}
\usepackage{amsmath}
\usepackage{amsthm}
\usepackage{amssymb, mathtools, hyperref}
\usepackage{graphicx}
\usepackage{cleveref}
\usepackage{tikz}
\usetikzlibrary{patterns,shapes}
\usepackage{gastex}
\usepackage{todonotes}
\usepackage{cite}

\def\N{\mathbb N}
\def\A{\mathcal A}
\def\B{\mathcal B}

\def\LL{\mathcal L}

\def\R{\mathcal R}

\def\uu{\mathbf{u}}
\def\vv{\mathbf{v}}
\def\xx{\mathbf{x}}
\def\yy{\mathbf{y}}
\def\ww{\mathbf{w}}

\def\dd{\mathbf{d}}

\def\gap{\mathrm{gap}}
\def\CR{\mathrm{E}}

\def\per{\mathrm{Per}}
\def\barva{\mathrm{colour}}

\theoremstyle{definition}
\newtheorem{definition}{Definition}
\newtheorem{corollary}[definition]{Corollary}
\newtheorem{remark}[definition]{Remark}
\newtheorem{example}[definition]{Example}

\theoremstyle{plain}
\newtheorem{theorem}[definition]{Theorem}
\newtheorem{proposition}[definition]{Proposition}
\newtheorem{lemma}[definition]{Lemma}

\newtheorem{conjecture}[definition]{Conjecture}

\begin{document}
\begin{frontmatter}
\title{On minimal critical exponent of balanced sequences}

\author[cvut]{L\!'ubom\'ira Dvo\v r\'akov\'a}
\ead{lubomira.dvorakova@fjfi.cvut.cz}
\author[cvut]{Edita Pelantov\'a}
\ead{edita.pelantova@fjfi.cvut.cz}
\author[cvut]{Daniela Opo\v censk\'a}
\ead{opocedan@fjfi.cvut.cz}
\author[ufn]{Arseny~M.~Shur}
\ead{arseny.shur@urfu.ru}
\address[cvut]{FNSPE Czech Technical University in Prague, Czech Republic}
\address[ufn]{Ural Federal University, Ekaterinburg, Russia}

\begin{abstract}
We study the threshold between avoidable and unavoidable repetitions in infinite balanced sequences over finite alphabets. The conjecture stated by Rampersad, Shallit and Vandomme says that the minimal critical exponent of balanced sequences over the alphabet of size $d \geq 5$ equals $\frac{d-2}{d-3}$. This conjecture is known to hold for $d\in \{5, 6, 7,8,9,10\}$. We refute this conjecture by showing that the picture is different for bigger alphabets. We prove that critical exponents of balanced sequences over an alphabet of size $d\geq 11$ are lower bounded by $\frac{d-1}{d-2}$ and this bound is attained for all even numbers $d\geq 12$. According to this result, we conjecture that the least critical exponent of a balanced sequence over $d$ letters is $\frac{d-1}{d-2}$ for all $d\geq 11$.

\end{abstract}

\end{frontmatter}

\section{Introduction}

The birth of combinatorics on words is linked to the study of repetitions of factors in infinite words (or sequences) by a Norwegian mathematician Axel Thue in 1906 \citep{Thue06}. He answered affirmatively the following two questions: Is there a binary sequence without cubes? Is there a ternary sequence without squares? In \citep{Thue12} he constructed the famous Thue--Morse sequence which is not only cube-free, but even overlap-free. Squares, cubes and overlaps are particular cases of \emph{fractional powers}. A word $w$ is a fractional power of a word $u$ if an infinite repetition $uuu\cdots$ begins with $w$. The ratio of lengths of $w$ and the shortest possible $u$ is the \emph{exponent} of $w$. The supremum of the exponents of all non-empty factors occurring in a sequence  $\uu$ is the \emph{critical} (or \emph{local}) exponent of $\uu$. The critical exponent of sequences and related questions of repetition avoidance has become today a classic area of combinatorics on words.

Obviously, the larger alphabet the smaller critical exponent can be found  among  the sequences over the alphabet. 
In 1988, Carpi \citep{Car1988} showed that for every real number $\alpha > 1$ there exists $d$ and a $d$-ary pure morphic sequence  with the critical exponent less than $\alpha$. Krieger and Shallit \citep{KrSh07} proved that  every real number greater than 1 is a critical exponent of some sequence. 

The search for the minimal critical exponent of infinite sequences over an alphabet of a fixed  size resulted in a conjecture formulated by Dejean \citep{Dej72} in 1972. The conjecture  states that the infimum of critical exponents of $d$-ary sequences equals:
\begin{itemize}
    \item 2 for $d=2$;
    \item $7/4$ for $d=3$;
    \item $7/5$ for $d=4$; 
    \item $\frac{d}{d-1}$ for $d\geq 5$.
\end{itemize}
The conjecture had been proved step by step by many people 
\citep{Dej72, Pan84c, Mou92, MoCu07, Car07, CuRa11, Rao11}.

The least critical exponent is also studied for particular families of sequences.
Carpi and de Luca \citep{CaLu2000} found out that the minimal critical exponent of a Sturmian sequence is  $\frac{5+\sqrt{5}}{2}$, reached by the Fibonacci sequence. The least critical exponent for binary rich sequences was determined recently by Curie, Mol and Rampersad \citep{CMR20}. The minimal value $2+ \tfrac{\sqrt{2}}{2}$  is reached by a complementary symmetric Rote sequence. Shallit and Shur \citep{ShSh19} proved a number of results connecting factor complexity and critical exponent of sequences. For example, they established that the Thue-Morse sequence and the twisted Thue-Morse sequence have, respectively, the minimum and the maximum factor complexity over all binary overlap-free sequences; the minimal critical exponent of a binary sequence of factor complexity $2n$ is $5/2$; the set of ternary square-free sequences either has no sequence of minimum complexity, or the minimum is reached by the ternary Thue sequence. 

For some types of sequences,  formulae for computation of the critical exponent are known. 
Blondin-Mass\'e  et al. \citep{BBGL07} computed the critical exponent for generalized Thue-Morse sequences. The critical  exponent of complementary symmetric Rote sequences is computed in \citep{DvMePe20}. Justin and Pirillo \citep{JuPi02} gave a formula for the critical exponent of standard episturmian sequences which are fixed by a primitive morphism. Krieger \citep{Krieger2007} provided an algorithm to compute the critical exponent for sequences that are fixed points of non-erasing morphisms.

In this paper we focus on the least critical exponent of balanced sequences. Let us recall that a sequence over a finite alphabet is balanced if, for any two of its factors $u$ and $v$ of the same length, the number of occurrences of each letter in $u$ and $v$ differs by at most 1. Over a binary alphabet, aperiodic balanced sequences coincide with Sturmian sequences, as shown by Morse and Hedlund~\citep{MoHe40}. Hubert~\citep{Hubert00} provided a uniform construction of all $d$-ary aperiodic recurrent balanced sequences from Sturmian sequences. Recently, Rampersad, Shallit and Vandomme~\citep{RSV19} found balanced sequences with the least critical exponent over alphabets of size 3 and 4 and also conjectured that the minimal critical exponent of balanced sequences over a $d$-ary alphabet with $d \geq 5$ is $\frac{d-2}{d-3}$.  For $d \leq 10$, they defined the candidate sequences ${\bf x}_d$, obtained from Sturmian sequences with  quadratic slope, to reach this minimum. In \citep{DolceDP21}, an algorithm  for computing the  critical exponent of  balanced sequences  of this type is deduced. The conjecture of Rampersad, Shallit and Vandomme was confirmed for $d \leq 8$  in~\citep{Bar20, BaSh19} and for $d\in \{9,10\}$ in~\citep{DolceDP21}.

In this paper, we first show that for balanced sequences over a $d$-ary alphabet with $d \geq 11$, the critical exponent is greater than or equal to $\frac{d-1}{d-2}$. Then for every  even $d\geq 12$, we find a $d$-ary balanced sequence ${\bf x}_d$ having the critical exponent $\frac{d-1}{d-2}$. Again, each sequence ${\bf x}_d$  is derived from a Sturmian sequence with a quadratic slope. In particular, for $d\geq 14$ this Sturmian sequence is the Fibonacci sequence. 

As $\frac{d-1}{d-2}<\frac{d-2}{d-3}$, our result  refutes the conjecture by Rampersad, Shallit and Vandomme.  
We state as a new conjecture that the minimal critical exponent of balanced sequences equals $\frac{d-1}{d-2}$ for $d \geq 11$. Thus it remains to prove this conjecture for sequences over alphabets of odd size.

\section{Preliminaries}
\label{Section_Preliminaries}
An \textit{alphabet} $\A$ is a finite set of symbols called \textit{letters}.
A \textit{word} over $\A$ of \textit{length} $n$ is a string $u = u_0 u_1 \cdots u_{n-1}$, where $u_i \in \A$ for all $i \in \{0,1, \ldots, n-1\}$.
The length of $u$ is denoted by $|u|$. 
The set of all finite words over $\A$ together with the operation of concatenation forms a monoid, denoted $\A^*$.
Its neutral element is the \textit{empty word} $\varepsilon$ and we denote $\A^+ = \A^* \setminus \{\varepsilon\}$.
If $u = xyz$ for some $x,y,z \in \A^*$, then $x$ is a \textit{prefix} of $u$, $z$ is a \textit{suffix} of $u$ and $y$ is a \textit{factor} of $u$.
To any word $u$ over $\A$ with cardinality $\#\A = d$,  we assign its \textit{Parikh vector} $\vec{V}(u) \in \N^{d}$ defined as $(\vec{V}(u))_a = |u|_a$ for all $a \in \A$, where $|u|_a$ is the number of letters $a$ occurring in $u$.

A \textit{sequence} over $\A$ is an infinite string $\uu = u_0 u_1 u_2 \cdots$, where $u_i \in \A$ for all $i \in \N$. The notation $\A^{\mathbb{N}}$ stands for the set of all sequences over $\A$. 
We always denote sequences by bold letters. 
The shift operator $\sigma$ maps any sequence $\uu = u_0 u_1 u_2 \cdots$ to the sequence $\sigma(\uu) = u_1 u_2 u_3 \cdots$. The \emph{frequency} of a letter $a$ in a sequence $\uu$ is the limit $\rho_a(\uu)=\lim_{n\to\infty} \frac{|u_0\cdots u_{n-1}|_a}{n}$ if it exists.

A sequence $\uu$ is \textit{eventually periodic} if $\uu = vwww \cdots = v(w)^\omega$ for some $v \in \A^*$ and $w \in \A^+$.
It is \textit{periodic} if $\uu=w^{\omega}$. In both cases, the number $|w|$ is a \textit{period} of $\uu$. We write $\per(\uu)$ for the minimal period of $\uu$.
If $\uu$ is not eventually periodic, then it is \textit{aperiodic}.
A \textit{factor} of $\uu = u_0 u_1 u_2 \cdots$ is a word $y$ such that $y = u_i u_{i+1} u_{i+2} \cdots u_{j-1}$ for some $i, j \in \N$, $i \leq j$. 
The number $i$ is called an \textit{occurrence} of the factor $y$ in $\uu$.
In particular, if $i = j$, the factor $y$ is the empty word $\varepsilon$ and any index $i$ is its occurrence.
If $i=0$, the factor $y$ is a \textit{prefix} of $\uu$.
If each factor of $\uu$ has infinitely many occurrences in $\uu$, the sequence $\uu$ is \textit{recurrent}.
Moreover, if for each factor the distances between its consecutive occurrences are bounded, $\uu$ is \textit{uniformly recurrent}. In a uniformly recurrent sequence all letters have frequencies.

The \textit{language} $\mathcal{L}(\uu)$ of a sequence $\uu$ is the set of all its factors.
% We also define $\mathcal{L}(\uu)^+ = \mathcal{L}(\uu) \setminus \{ \varepsilon \}$.
A factor $w$ of $\uu$ is \textit{right special} if $wa, wb$ are in $\mathcal{L}(\uu)$ for at least two distinct letters $a,b \in \A$.
A \textit{left special} factor is defined symmetrically.
A factor is \textit{bispecial} if it is both left and right special.
%Note that the empty word $\varepsilon$ is bispecial if at least two distinct letters occur in $\uu$.
The \textit{factor complexity} of a sequence $\uu$ is the mapping $\mathcal{C}_\uu: \N \to \N$ defined by
$\mathcal{C}_\uu(n) = \# \{w \in \mathcal{L}(\uu) : |w| =  n \}$.
% The first difference of the factor complexity is $s_{\uu}(n) = \mathcal{C}_{\uu}(n+1) - \mathcal{C}_{\uu}(n)$.

The factor complexity of an aperiodic sequence $\uu$ satisfies $\mathcal{C}_{\uu}(n) \ge n+1$ for all $n \in \N$. The aperiodic sequences with the lowest possible factor complexity $\mathcal{C}_{\uu}(n) = n+1$ are called \textit{Sturmian sequences}.
Clearly, all Sturmian sequences are defined over a binary alphabet, e.g., $\{ {\tt a,b} \}$.
%Note that in this case $s_{\uu}(n) = 1$ for every $n \in \mathbb{N}$.

A~sequence $\uu\in \A^{\mathbb{N}}$ is \textit{balanced} if for every letter $a \in \A$ and every pair of factors $u,v \in {\mathcal L}(\uu)$ with $|u|=|v|$, we have $|u|_a-|v|_a\leq 1$. The class of Sturmian sequences and the class of aperiodic balanced sequences coincide over a~binary alphabet (see~\citep{MoHe40}). Every recurrent balanced sequence is uniformly recurrent (see \citep{DolceDP21}). 
% Vuillon~\citep{Vui} provides a survey on some previous work on balanced sequences.

A \textit{morphism} over $\A$ is a mapping $\psi: \A^* \to \A^*$ such that $\psi(uv) = \psi(u)\psi(v)$  for all $u, v \in \A^*$.
Morphisms can be naturally extended to $\A^{\mathbb{N}}$ by setting
$\psi(u_0 u_1 u_2 \cdots) = \psi(u_0) \psi(u_1) \psi(u_2) \cdots\,$.
A \textit{fixed point} of a morphism $\psi$ is a sequence $\uu$ such that $\psi(\uu) = \uu$.
%The \textit{matrix} of a morphism $\psi$ over $\A$ with the cardinality $\#\A = d$ is the matrix $M_\psi \in \N^{d\times d}$ defined as $(M_\psi)_{a,b} = |\psi(a)|_b$ for all $a,b \in \A$.
%The Parikh vector of the $\psi$-image of a word $w\in \A^*$ can be obtained via multiplication by the matrix $M_\psi$, i.e., $\vec{V}({\psi(w)}) = M_\psi \vec{V}(w)$.

Consider a factor $w$ of a recurrent sequence $\uu = u_0 u_1 u_2 \cdots$.
Let $i < j$ be two consecutive occurrences of $w$ in $\uu$.
Then the word $u_i u_{i+1} \cdots u_{j-1}$ is a \textit{return word} to $w$ in $\uu$.
The set of all return words to $w$ in $\uu$ is denoted by $\R_\uu(w)$.
If $\uu$ is uniformly recurrent, the set $\R_\uu(w)$ is finite for each factor $w$.
If $k$ is the first occurrence of a factor $w$ in $\uu$, then the $k$-th shift of $\uu$ can be written as a concatenation $\sigma^{k}(\uu)=r_{d_0}r_{d_1}r_{d_2} \cdots$ of return words to $w$. In particular, if $w$ is a prefix of $\uu$, then $\uu=r_{d_0}r_{d_1}r_{d_2} \cdots$. 
The sequence
$\dd_\uu(w) = d_0d_1d_2 \cdots$
 over the alphabet of cardinality $\# \R_\uu(w)$ is called the  \textit{derived sequence} of $\uu$ to $w$.
 The concept of derived sequences was introduced by Durand~\citep{Dur98}.

For an arbitrary nonempty word $z$, let $u$ be the shortest word such that $z$ is a prefix of the periodic sequence $u^\omega$. The number $|u|$ is the (minimal) \emph{period} of $z$, and the ratio $e=|z|/|u|$ is the \emph{exponent} of $z$, written as $e=\exp(z)$. The \emph{critical exponent} of an infinite sequence $\uu$ is defined as 
$$ 
\CR(\uu)  = \sup \{ \exp(z) : z \ \text{is a non-empty factor of} \ \uu\}\,. 
$$
The critical exponent of a uniformly recurrent sequence\footnote{In fact, the same result holds for all aperiodic recurrent sequences. As Theorem~\ref{Prop_FormulaForCR} is sufficient for our purposes, we do not prove the more general result here.} can be computed from its bispecial factors and their return words:
\begin{theorem}[{\citep{DDP21}}]
\label{Prop_FormulaForCR}
Let $\uu$ be a uniformly recurrent aperiodic sequence.
Let $(w_n)$ be a sequence of all bispecial factors of $\uu$, ordered by their length.
For every $n \in \N$, let $v_n$ be a shortest return word to $w_n$ in $\uu$.
Then
$$
\CR(\uu) = 1 + \sup\limits_{n \in \N} \left\{ \frac{|w_n|}{|v_n|} \right\}.
$$
\end{theorem}

\subsection{Sturmian sequences}
\label{sec:sturmian}

Sturmian sequences are a principal tool in the study of balanced sequences over arbitrary alphabets. In this section we recall the necessary facts about them.

We recall that Sturmian sequences can be considered as \emph{cutting sequences} of straight lines with irrational slopes \citep{MoHe40}. The definition is as follows. Consider the positive quadrant of the coordinate plane and a square grid on it, parallel to the axes (the axes themselves do not belong to the grid). The intersection of a straight line with the grid can be encoded as a binary sequence: symbol $\mathtt{a}$ (resp., $\mathtt{b}$) encodes the intersection with the horizontal (resp., vertical) line of the grid. Thus, each Sturmian sequence $\uu$ has an irrational \emph{slope} $\theta(\uu)$ which is the slope of the straight line producing $\uu$ as a cutting sequence. From the definition of the cutting sequence it is clear that the letter frequencies in a Sturmian sequence are irrational, and  the slope of a Sturmian sequence equals the ratio of these frequencies. 

All Sturmian sequences with the same slope share the same language. Among them, there is a unique \textit{standard sequence}, which is the cutting sequence of a line intersecting the origin. Equivalently, the standard sequence can be defined by the condition that both sequences ${\tt a}\uu$ and ${\tt b}\uu$ are Sturmian.

\begin{example}
\label{ex:FiboDef}
The most famous standard sequence is  the \textit{Fibonacci sequence}
$$
\uu_f = {\tt babbababbabbababbababb}\cdots\,,
$$
defined as the fixed point of the morphism $f: {\tt b} \mapsto {\tt ba}$, ${\tt a} \mapsto {\tt b}$. Its slope is $\varphi=\frac{\sqrt{5}-1}{2}\approx 0.618$ (the inverse of the golden ratio), the frequencies of $\tt a$ and $\tt b$ are $\varphi^2$ and $\varphi$ respectively, and the critical exponent of $\uu_f$ is $3+\varphi$, which is the minimum among Sturmian sequences.
\end{example}

We use the characterization of standard sequences by their directive sequences. To introduce them, we define the two morphisms
$$
G = \
\left\{ \,
\begin{aligned}
{\tt a} & \to {\tt a} \\
{\tt b} & \to {\tt ab} \,
\end{aligned}
\right.
\quad \text{and} \quad
D = \
\left\{ \,
\begin{aligned}
{\tt a} & \to {\tt ba} \\
{\tt b} & \to {\tt b} \,
\end{aligned}
\right. .
$$

\begin{proposition}[{\citep{JuPi02}}]
\label{Lem_Standard}
For every standard sequence $\uu$ there is a uniquely given \emph{directive sequence} $\mathbf{\Delta} = \Delta_0 \Delta_1 \Delta_2 \cdots \in \{ G, D \}^\N$ of morphisms and a sequence $(\uu^{(n)})$ of standard sequences such that
$$
\uu = {\Delta_0 \Delta_1 \ldots \Delta_{n-1}} \left( \uu^{(n)} \right) \, \ \text{for every } \ n \in \N\,.
$$
Both $G$ and $D$ occur in the sequence $\mathbf{\Delta}$ infinitely often.
\end{proposition}

If $\Delta_0 = D$, then by Proposition~\ref{Lem_Standard} $\uu$ is the image of a standard sequence under the morphism $D$ and consequently, $\mathtt{b}$ is the most frequent letter in $\uu$. Otherwise,  $\mathtt{a}$ is the most frequent letter in $\uu$.  We adopt the convention that $\rho_{\tt b}(\uu) > \rho_{\tt a}(\uu)$ and thus the directive sequence of $\uu$ starts with $D$. Let us write this sequence in the run-length encoded form  $\mathbf{\Delta} = D^{a_1} G^{a_2} D^{a_3} G^{a_4} \cdots$,
where all integers $a_n$ are positive. Then the number $\theta$ having the continued fraction expansion 
$\theta = [0, a_1, a_2, a_3, \ldots]$ equals the ratio 
$\frac{\rho_{\tt a}(\uu)}{\rho_{\tt b}(\uu)}$ (see \citep{BeSe02}) and thus $\theta=\theta(\uu)$.

Knowing the coefficients of the continued fraction of the slope of $\uu$, one can find prefixes of $\uu$. The more initial coefficients of $\theta(\uu)$ we have, the longer prefixes of $\uu$ we can reconstruct. In particular, if the directive sequences of $\uu$ starts with $D^{a_1}$, then $\uu=\mathtt{b}^{a_1}\mathtt{a}\cdots$. 

\begin{example}\label{ex:factors} 
Consider the standard sequence $\uu$ with the slope $\theta=[0,1,3,\overline{2}]$.  As any sequence of the form $G^2(\ww)$ starts with $\texttt{aab}$, the following word is a prefix of $\uu$:
\begin{multline*}
DG^3D^2(\mathtt{aab}) = DG^3(\mathtt{bbabbab}) = D(\mathtt{a^3ba^3ba^4ba^3ba^4b}) =\\
= \mathtt{bababab}\,\mathtt{bababab}\,\mathtt{babababab}\,
\mathtt{bababab}\, \mathtt{babababab}.
\end{multline*}
\end{example}
The convergents  to the continued fraction of $\theta$,  usually  denoted $\frac{p_N}{q_N}$,   and their secondary convergents have a close relation to the return words in a Sturmian sequence. Recall that the sequences $(p_N)$ and $(q_N)$ both satisfy the recurrence relation 
\begin{equation}\label{eq:convergents}
X_{N+1}=a_{N+1}X_N+X_{N-1}
\end{equation}
with initial conditions $p_{-1}=1$, $p_0=0$ and $q_{-1}=0$, $q_0=1$.  Two consecutive convergents satisfy $p_{N}q_{N-1}-p_{N-1}q_N = (-1)^{N+1}$ for every $N\in \N$.

Vuillon~\citep{Vui01} showed that an infinite recurrent sequence $\uu$ is Sturmian if and only if each of its factors has exactly two return words. Moreover, the derived sequence of a Sturmian sequence to any its factor is also Sturmian. 

All bispecial factors of any standard sequence $\uu$ are its prefixes. So, one of the return words to a bispecial factor of $\uu$ is a prefix of $\uu$.

\begin{proposition}[{\citep{DvMePe20}}]\label{prop:returnWords}
\label{ParikhRSB1}
Suppose that $\uu$ is a standard sequence with the slope $\theta = [0, a_1,a_2, a_3, \ldots]$ and $z$ is a bispecial factor of $\uu$. Let $r$ (resp., $s$) denote the return word to $z$ which is (resp., is not) a prefix of $\uu$. Then
\begin{enumerate}
\item there exists a unique pair $(N,m) \in \N^2$ with $0 \le m < a_{N+1}$ such that the Parikh vectors of $r$, $s$, and $z$ are respectively
$$
\vec{V}(r) = \begin{pmatrix} p_N \\ q_N \end{pmatrix},
\; \; 
\vec{V}(s) = \begin{pmatrix} m \, p_{N} + p_{N-1} \\ m \, q_{N} + q_{N-1} \end{pmatrix},
\; \; 
\vec{V}(z) = \vec{V}(r) + \vec{V}(s) - \begin{pmatrix} 1 \\ 1 \end{pmatrix};
$$
\item the slope of the derived sequence $\dd_\uu(z)$ is
$${\theta}'=[0, a_{N+1} - m, a_{N+2}, a_{N+3}, \ldots].$$
\end{enumerate}
\end{proposition}

% \begin{remark}
% If we order by their length the bispecial factors of $\uu$ associated to $\theta = [0, a_1,a_2, a_3, \ldots]$, then the pair $(N,m)$ associated to the $n$-th bispecial is given by $n=a_0+a_1+\dots+a_N+m$, where $a_0=0$ and $0\leq m <a_{N+1}$.
% \end{remark} \todo{L: Remark may be erased.}
\begin{lemma}[{\citep{DDP21}}]
\label{lem_kl}
Let $\uu$ be a Sturmian sequence with the slope $\theta=\frac{\rho_{\tt a}(\uu)}{\rho_{\tt b}(\uu)}<1$.
Then $\uu$ contains a factor $w$ such that $|w|_{\tt b} = k$ and $|w|_{\tt a} = \ell$ if and only if
\begin{equation}
\label{pocet}
(k-1) \theta - 1 < \ell < (k+1) \theta + 1 \ \ \text{and} \ k, \ell \in \N.
\end{equation}
\end{lemma}

\subsection{Balanced sequences}
\label{Section_BalancedSturmian1}

In 2000 Hubert~\citep{Hubert00} characterized balanced sequences over alphabets of cardinality bigger than 2 in terms of Sturmian sequences, \textit{colourings}, and \textit{constant gap sequences}.

\begin{definition}\label{def:colouring}
Let $\uu$ be a sequence over $\{\mathtt{a},\mathtt{b}\}$, $\yy$ and $\yy'$ be arbitrary sequences. The \emph{colouring}\footnote{This operation is also known as \textit{shuffling} of $\yy$ and $\yy'$ with \textit{directive sequence} $\uu$} of $\uu$ by $\yy$ and $\yy'$ is the sequence $\vv = \barva( \uu, \yy, \yy')$ obtained from $\uu$ by replacing the subsequence of all $\mathtt{a}$'s with $\yy$ and the subsequence of all $\mathtt{b}$'s with $\yy'$.
%Let $\uu$ be a Sturmian sequence over the alphabet $\{ {\tt a}, {\tt b}\}$, and ${\yy}, {\yy'}$ be two constant gap sequences over two disjoint alphabets $\A$ and $\B$.
%Let a~sequence $\vv$ be obtained from $\uu$ by replacing the ${\tt a}$'s in $\uu$ step by step by letters of $\yy$ and replacing the ${\tt b}$'s in $\uu$ step by step by letters of $\yy'$. The sequence $\vv$ is called \emph{colouring} of $\uu$ by $\yy$ and $\yy'$ and denoted $\vv = \barva( \uu, \yy, \yy')$.
\end{definition}

%A suitable tool for their description is the notion of constant gap.

\begin{definition}
\label{Def_ConstantGap}
A~sequence $\yy$ is a \emph{constant gap sequence} if for each letter $a$ occurring in $\yy$ there is a positive integer denoted by $\gap_{\yy}(a)$ such that the distance between any consecutive occurrences of $a$ in $\yy$ is  ${\rm gap}_{\yy}(a)$.
\end{definition}

Obviously, every constant gap sequence $\yy$ is periodic and $\per(\yy)$ is the least common multiple of all numbers $\gap_{\yy}(a)$.

\begin{example}
\label{ex:sequences}
The sequence $\yy = (\mathtt{0102})^\omega$ is a constant gap sequence because the distance between consecutive $\mathtt{0}$'s is always $2$, while the distance between consecutive $\mathtt{1}$'s (resp., $\mathtt{2}$'s) is always $4$. Its minimal period is $\rm{Per}(\yy) = 4$. 
% For factors of length 2 we have ${\rm gap}_{\yy}(\tt 01)={\rm gap}_{\yy}(\tt 10)={\rm gap}_{\yy}(02)={\rm gap}_{\yy}(\tt 20)=4$.  
 
The sequence $\vv = \barva( \uu_f, (\texttt{AB})^\omega, (\texttt{0102})^\omega)$, where $\uu_f$ is defined in Example~\ref{ex:FiboDef}, looks as follows:
\begin{align*}
\uu_f &= \mathtt{babbababbabbababbababbabbababba}\cdots \\
\vv &= \mathtt{ \textcolor{red}{0}A\textcolor{red}{10}B\textcolor{red}{2}A\textcolor{red}{01}B\textcolor{red}{02}A\textcolor{red}{0}B\textcolor{red}{10}A\textcolor{red}{2}B\textcolor{red}{01}A\textcolor{red}{02}B\textcolor{red}{0}A\textcolor{red}{10}B} \cdots
\end{align*}
\end{example}

\begin{theorem}[{\citep{Hubert00}}]
\label{Hubert}
A recurrent aperiodic sequence $\vv$ is balanced if and only if $\vv = \barva( \uu, \yy, \yy')$ for some Sturmian sequence $\uu$ and constant gap sequences ${\yy}, {\yy'}$ over two disjoint alphabets.
\end{theorem}

\begin{corollary} \label{c:irratfreq}
A letter $a$ from $\yy$ (resp., $b$ from $\yy'$) occurs in $\vv = \barva( \uu, \yy, \yy')$ with frequency $\rho_a(\vv)=\frac{\rho_{\tt a}(\uu)}{\gap_{\yy}(a)}$ (resp., $\rho_b(\vv)=\frac{\rho_{\tt b}(\uu)}{\gap_{\yy}(b)}$). In particular, all frequencies of letters in aperiodic balanced sequences are irrational.
\end{corollary}

\noindent\textit{Example~\ref{ex:sequences} (continued).}
By Theorem~\ref{Hubert}, the sequence $\vv$ is balanced. Knowing the frequencies of letters in $\uu_f$ from Example~\ref{ex:FiboDef}, we can compute the frequencies of letters in $\vv$ by Corollary~\ref{c:irratfreq}; e. g., $\rho_{\tt 0}(\vv)= \varphi/2$.

\medskip
Let $\A,\B$ be two disjoint alphabets. The ``discolouration map'' $\pi$ is defined for any word or sequence over $\A\cup \B$; it replaces all letters from $\A$ by $\mathtt{a}$ and all letters from $\B$ by $\mathtt{b}$. If  $\vv = \barva( \uu, \yy, \yy')$, where $y\in \A^{\mathbb{N}}$, $y'\in \B^{\mathbb{N}}$, then $\pi(\vv)=\uu$ and $\pi(v) \in \mathcal{L}(\uu)$ for every $v \in \LL(\vv)$.

% \todo{L: The following Proposition and Corollary may be erased when we keep Arseny's original proof.}
% Let us recall some properties of balanced sequences shown in~\citep{DlouhyClanek}. The next proposition says that for a balanced sequence $\vv= \barva( \uu, \yy, \yy')$, the language $\LL(\vv)$ depends on $\LL(\uu)$ but not on the Sturmian sequence $\uu$ itself. Furthermore, $\LL(\vv)$ is invariant under shifts of the sequences $\yy$ and $\yy'$.

% \begin{proposition} [\citep{DlouhyClanek}]
% \label{pro:standardSturm}
% Let $\uu,\uu'$ be Sturmian sequences such that $\LL(\uu) = \LL(\uu')$, $\yy$ and $\yy'$ be constant gap sequences over two disjoint alphabets and $i,j \in \N$.
% If $\vv = \barva( \uu, \yy, \yy')$, $\vv' = \barva( \uu', \yy, \yy')$ and $\vv'' = \barva( \uu, \sigma^i(\yy), \sigma^j(\yy'))$, then $\LL(\vv) = \LL(\vv') = \LL(\vv'')$.
% \end{proposition}

% As a consequence of Proposition~\ref{pro:standardSturm} we get another invariance of the language of a balance sequence.

% \begin{corollary}[\citep{DlouhyClanek}]
% \label{cor:cyclicshift}
% Let $\vv = \barva( \uu, \yy, \yy')$ and $u \in \LL(\uu)$.
% For any non-negative integers $i,j$, the word $v$ obtained from $u$ by replacing the $\mathtt{a}$'s by a prefix of $\sigma^i(\yy)$ and the $\mathtt{b}$'s by a prefix of $\sigma^j(\yy')$ belongs to  $\LL(\vv)$.
% \end{corollary}

\section{Lower Bounds on Critical Exponent}
First we prove a simple but useful property.

\begin{lemma} \label{l:distance}
Let $\vv$ be an aperiodic balanced sequence, $a$ be a letter in $\vv$. The set of distances between consecutive occurrences of $a$ in $\vv$ consists of two consecutive integers.
\end{lemma}
\begin{proof}
Since the frequency of $a$ in $\vv$ is irrational by Corollary~\ref{c:irratfreq}, there should be at least two different distances between consecutive occurrences of $a$. Let $k$ be the minimal such distance. Then $\vv$ has a factor of length $k{+}1$ with two $a$'s. By the balance property, each factor of this length contains the letter $a$, so the distance between consecutive $a$'s cannot exceed $k{+}1$. Hence the set of distances is $\{k,k{+}1\}$, as required. 
\end{proof}

Now we state the main result of this section.

\begin{theorem}\label{t:lbound}
For each $d \ge 11$, there exists no $d$-ary balanced sequence $\vv$ with $\mathrm{E}(\vv)<\frac{d-1}{d-2}$.
\end{theorem}

Let us fix an arbitrary alphabet $\A$ with $d\ge 11$ letters and introduce the necessary tools\footnote{The argument below works for $d\ge 4$, but since for $d\le 10$ a stronger lower bound is known \citep{RSV19}, only the case $d\ge 11$ is of interest.}. Any sequence $\vv\in\A^{\mathbb{N}}$ with $\mathrm{E}(\vv)<\frac{d-1}{d-2}$ satisfies the following ``local'' properties:
\begin{itemize}
    \item[(i)] in each factor of $\vv$ of length $d-1$, all letters are distinct;
    \item[(ii)] any two consecutive occurrences of the same letter in $\vv$ are followed by different letters.
\end{itemize}
Property (i) is obvious. If (ii) fails, then $\vv$ has a factor $abXab$, where the word $X$ does not contain the letter $a$. Assume $X$ contains $b$: $X=YbZ$. By (i), $|bYb|, |bZab|\ge d$. Then $|bX|\ge 2d-3$, and $bX$ contains only $d-1$ distinct letters. Again by (i), $bX$ has the period $d-1$ and thus $\exp(bX)\ge \frac{2d-3}{d-1}$, which is impossible. Therefore $X$ contains neither of the letters $a,b$. Then  $|X|\le d-2$ by (i), implying 
\begin{equation} \label{eq:dplus2}
\mathrm{E}(\vv)\ge \exp(abXab)=\frac{|abXab|}{|abX|}\ge \frac{d+2}{d}> \frac{d-1}{d-2}\, ,   
\end{equation}
which is impossible. Thus, (ii) holds.

Following \citep{ShGo10}, we refer to the words and sequences satisfying (i) and (ii) as \emph{Pansiot words/sequences}. Note that (i) and (ii) imply that a factor of length $d+1$ should contain $d$ distinct letters. As a result, Pansiot words/sequences satisfy the property
\begin{itemize}
    \item[(iii)]  any two consecutive occurrences of the same letter are at the distance $d-1$, $d$, or $d+1$.
\end{itemize}

A handy tool for studying Pansiot words and sequences is the \emph{cylindric representation} introduced in \citep{ShGo10}. A  Pansiot word/sequence can be viewed as a rope with knots representing letters. This rope is wound around a cylinder such that the knots at distance $d$ are placed one under another. A part of a projection of such a cylinder is drawn in Fig.~\ref{f:cyl}, a. By (iii), the knots labeled by two consecutive occurrences of the same letter appear on two consecutive winds of the rope one under another or shifted by one knot (Fig.~\ref{f:cyl}, b). Connecting consecutive occurrences by line segments (``sticks''), we get three types of sticks: vertical, left-slanted, and right-slanted (Fig.~\ref{f:cyl},~b); these three types correspond respectively to the distances $d$, $d+1$, and $d-1$ between the occurrences. The sticks form $d$ broken lines, one per letter; we call these lines \emph{traces} of letters.

 \begin{figure}[htb]
     \centering
     \includegraphics[trim = 49 639 220 34, clip]{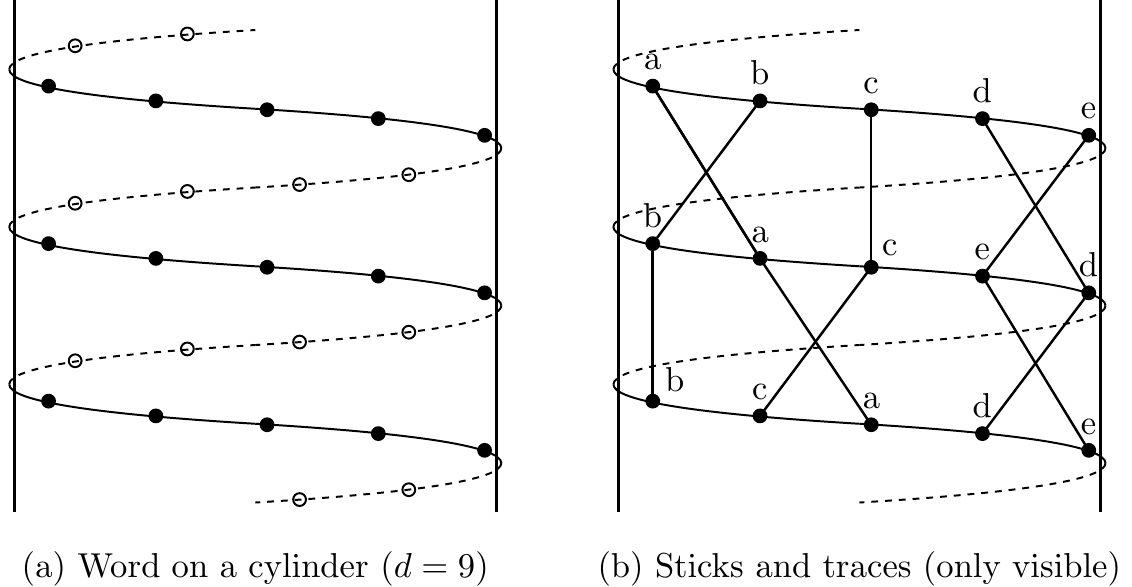}
     \caption{Cylindric representation of a Pansiot word. }
     \label{f:cyl}
 \end{figure}

 From (ii), we derive the following property:
 \begin{itemize}
     \item[(iv)] any two subsequent sticks in the cylindric representation of a Pansiot word (sequence) are distinct.
 \end{itemize}

From (iii) and Lemma~\ref{l:distance} we have:
 \begin{itemize}
     \item[$(\star)$] the trace of a letter in the cylindric representation of a balanced Pansiot sequence contains either no left-slanted sticks or no right-slanted sticks.
 \end{itemize}
(For example, if the cylindric representation of $\vv$ contains a fragment shown in Fig.~\ref{f:cyl}, b, then $\vv$ is not balanced.) 
We call a letter  \emph{frequent} (in $\vv$) if it has no consecutive occurrences at distance $d+1$ and \emph{rare} (in $\vv$) if it has no consecutive occurrences at distance $d-1$. The following key lemma shows that the distance $d$ is possible only for one of these two classes of letters.

 \begin{lemma} \label{l:traces}
In a balanced Pansiot sequence, either the traces of all frequent letters consist only of right-slanted sticks or the traces of all rare letters consist only of left-slanted sticks.
 \end{lemma}

\begin{proof}
Let us fix a balanced Pansiot sequence $\vv=v_0v_1v_2\cdots$  and assume that the trace of some frequent letter $a$ in $\vv$ contains a vertical stick (otherwise, there is nothing to prove). Consider a fragment of the cylindric representation of $\vv$ around a fixed vertical stick in the trace of $a$ (Fig.~\ref{f:sticks}, a--d). By (iv), this stick (the blue one) is surrounded by pairs of crossed slanted sticks (Fig.~\ref{f:sticks}, a). Note that two vertical sticks cannot have a common knot: this contradicts $(\star)$ (Fig.~\ref{f:sticks}, b). Since $a$ is frequent, two more crossed pairs should be added to the picture (Fig.~\ref{f:sticks}, c). Next, the traces of the letters $c$ and $d$ contain right-slanted sticks, so these letters are frequent. Then the trace of $c$ (resp., of $d$) extends up (resp., down) by a vertical stick (Fig.~\ref{f:sticks}, d). Thus we proved the following fact: if, for some $i$, $v_i=v_{i+d}$ is a frequent letter, then $v_{i-d-1}=v_{i-1}$ is a frequent letter and $v_{i+d+1}=v_{i+2d+1}$ is also a frequent letter (these three pairs of equal letters correspond to three vertical sticks in Fig.~\ref{f:sticks}, d). Thus 
\begin{itemize}
     \item[$(\dagger)$] there exists $p\in \{0, \dots, d\}$ such that all positions equal to $p$ modulo $d{+}1$ are occupied in $\vv$ by frequent letters and correspond to vertical sticks.
\end{itemize}

\begin{figure}[htb]
     \centering
     \includegraphics[trim = 43 730 255 29, clip]{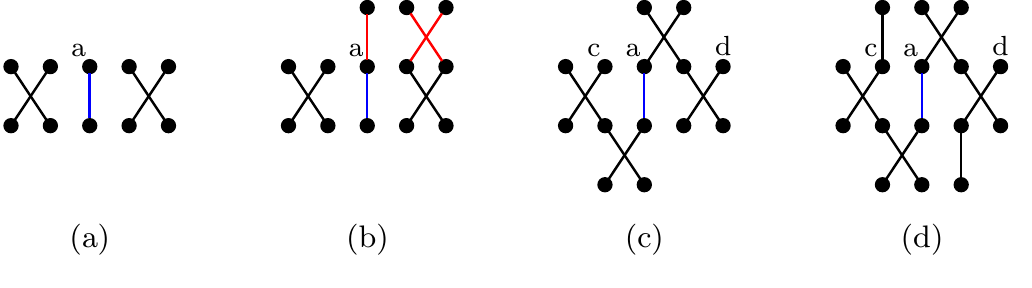}\\
     \includegraphics[trim = 43 730 330 26, clip]{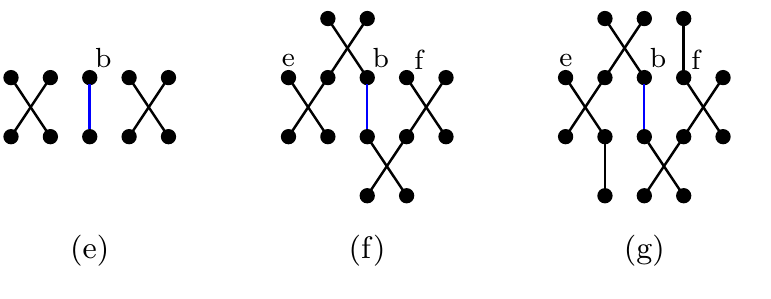}
     \caption{Lemma~\ref{l:traces}: mutual location of vertical sticks in the traces of frequent and rare letters. }
     \label{f:sticks}
\end{figure}

Now assume that the trace of some rare letter $b$ in $\vv$ contains a vertical stick (or there is nothing to prove). By the same argument as above, we reconstruct a fragment of the cylindric representation of $\vv$ (Fig.~\ref{f:sticks}, e-g). As a result, we get the following analog of $(\dagger)$: 
 \begin{itemize}
     \item[$(\ddagger)$] there exists $q\in \{0, \dots, d{-}2\}$ such that all positions equal to  $q$ modulo $d{-}1$ are occupied in $\vv$ by rare letters and correspond to vertical sticks.
 \end{itemize}
 It remains to note that $(\dagger)$ and $(\ddagger)$ cannot hold simultaneously. Indeed, if $p$ and $q$ have the same parity, then some letter should be both frequent and rare in $\vv$, which is impossible (see Fig.~\ref{f:sticks}, b); if $p$ and $q$ are of different parity, then the cylindric representation of $\vv$ contains two consecutive vertical sticks, which contradicts (ii). But as $(\dagger)$ and $(\ddagger)$ does not hold simultaneously, at least one of our assumptions about the existence of vertical sticks is false. The statement of the lemma is immediate from this.
 \end{proof}

\begin{proof}[Proof of Theorem~\ref{t:lbound}]
Let us take a $d$-ary aperiodic balanced sequence $\vv$ and assume $\mathrm{E}(\vv)<\frac{d-1}{d-2}$. Then $\vv$ is a Pansiot sequence.  By Lemma~\ref{l:traces}, either each frequent letter in $\vv$ has frequency $\frac{1}{d-1}$ or each rare letter in $\vv$ has frequency $\frac{1}{d+1}$. But these frequencies must be irrational by Corollary~\ref{c:irratfreq}. This contradiction proves our assumption false.
% the alphabet of $w$ can be partitioned into the set $\{a_1,\ldots,a_\ell\}$ of frequent letters and the set $\{b_1,\ldots,b_{k-\ell}\}$ of rare letters. The letters are ordered such that $i<j$ means that the first occurrence of $a_i$ (resp., of $b_i$) in $\ww$ precedes the first occurrence of $a_j$ (resp., $b_j$). Note that frequent letters never ``swap'' in $\ww$: property (iii) and the definition imply that if $n$th occurrence of $a_i$ precedes $n$th occurrence of $a_j$, then $(n{+}1)$th occurrence of $a_i$ precedes $(n{+}1)$th occurrence of $a_j$; the same property holds for rare letters. As a result, $\ww$ can be obtained by shuffling the periodic infinite words $(a_1\cdots a_\ell)^\infty$ and $(b_1\cdots b_{k-\ell})^\infty$.

% Now we use Lemma~\ref{l:traces}. Suppose that the traces of frequent letters contain only right-slanted tricks. Then the positions of each letter $a_i$ in $\ww$ form an arithmetic sequence with the difference $k-1$. Then $\ww=u^\infty$, where $u$ is the prefix of $\ww$ of length $(k-1)(k-\ell)$: each frequent (resp., rare) letter has $k-\ell$ (resp., $k-\ell-1$) occurrences in $u$. Symmetrically, if the traces of rare letters contain only left-slanted sticks, the positions of each $b_i$ form an arithmetic sequence with the difference $k+1$. In this case, $\ww=u^\infty$, where $u$ is the prefix of $\ww$ of length $(k+1)\ell$, and each frequent (resp., rare) letter has $\ell+1$ (resp., $\ell$) occurrences in $u$.

% As $\ww$ is periodic, it is not $\frac{k-1}{k-2}$-free. The theorem is proved.
\end{proof}

If the bound of Theorem~\ref{t:lbound} is tight, then the balanced sequences of minimal critical exponent should satisfy the following property.

\begin{proposition} \label{p:maxfreq}
Let $\vv$ be a balanced sequence over $d\ge 11$ letters such that $\mathrm{E}(\vv)=\frac{d-1}{d-2}$. Then the maximum letter frequency $\rho$ in $\vv$ satisfies $\frac{1}{d-1}< \rho < \frac{1}{d-2}$.
\end{proposition}

\begin{proof}
Every factor of length $d-2$ in $\vv$ contains no repeating letters. Hence the minimal distance between two occurrences of a letter is $d-2$. Assume that every factor of length $d-1$ $\vv$ contains no repeating letters. Then $\vv$ is a Pansiot sequence: it satisfies both (i) and (ii) (the negation of (ii) implies the inequality \eqref{eq:dplus2}, contradicting the condition $\mathrm{E}(\vv)=\frac{d-1}{d-2}$). As was shown in the proof of Theorem~\ref{t:lbound}, there are no aperiodic balanced Pansiot sequences, so our assumption is false. Therefore, $\vv$ contains factors of length $d-1$, having the form $aXa$. By Lemma~\ref{l:distance}, the distances between consecutive $a$'s in $\vv$ are $d-2$ and $d-1$. As the frequency of $a$ is irrational, the inequalities $\frac{1}{d-1}< \rho_a(\vv) < \frac{1}{d-2}$ are strict. 
\end{proof}

\section{Balanced Sequences Reaching the Lower Bound}
\label{s:sequences}

In this section we show that the lower bound  $\frac{d-1}{d-2}$ is attained for balanced sequences over a $d$-ary alphabet with $d=2\delta$ and $\delta \geq 6$. For each alphabet, the considered sequence $\xx_{2\delta}$ is the colouring by $\yy=(\tt 1 2 \cdots \delta)^{\omega}$ and $\yy'=(\tt 1' 2' \cdots \delta')^{\omega}$ of a standard sequence with the slope $\theta$ that will be specified later. We treat the cases $\delta \geq 7$ and $\delta=6$ separately, but we begin with the statements that are common to both cases. By Theorem \ref{Prop_FormulaForCR}, to prove that $E(\xx_{2\delta})=\frac{d-1}{d-2}$ it is sufficient to show that every bispecial factor $w \in \mathcal{L}(\xx_{2\delta})$ and its shortest return word $v$ satisfy  $\frac{|w|}{|v|}\leq \frac1{2\delta-2}$.  

Let $\xx_{2\delta}=\barva(\uu,\yy,\yy')$. There are only two distinct frequencies of letters in $\xx_{2\delta}$, namely, $\frac{\rho_{\tt a}(\uu)}{\delta}$ and $\frac{\rho_{\tt b}(\uu)}{\delta}$. By Proposition~\ref{p:maxfreq}, $\frac{1}{2\delta-1}< \frac{\rho_{\tt b}(\uu)}{\delta} < \frac{1}{2\delta-2}$. This can be converted into the double inequality $1-\frac{2}{\delta} < \theta < 1-\frac{1}{\delta}$ for the slope of $\uu$ (recall that $\theta=\frac{\rho_{\tt a}(\uu)}{\rho_{\tt b}(\uu)}$). Given this restriction, we define the slope to have the form $\theta=[0,1,\lfloor \delta/2 \rfloor,\ldots]$. Such a slope guarantees that $\xx_{2\delta}$ contains no short factors with the exponent greater than $\frac{d-1}{d-2}$, as the following proposition shows.

\begin{proposition} \label{p:short}
Suppose that $\xx_{2\delta}=\barva(\uu,\yy,\yy')$, where $\yy=(\tt 1 2 \dots \delta)^{\omega}$, $\yy'=(\tt 1' 2' \cdots \delta')^{\omega}$, and the slope of the Sturmian sequence $\uu$ has the form $\theta=[0,1,\lfloor \delta/2 \rfloor,\ldots]$. Then
\begin{enumerate}
    \item the distance between occurrences of a letter in $\xx_{2\delta}$ is at least $2\delta-2$;
    \item the distance between occurrences of a length-2 factor, both letters of which are from $\yy'$, in $\xx_{2\delta}$, is greater than $4\delta-4$. 
\end{enumerate} 
\end{proposition}

\begin{proof}
The given prefix of the continued fraction expansion of $\theta$ allows us to conclude that $\uu$ is an infinite concatenation of the ``blocks'' $(\tt ba)^{\lfloor \delta/2 \rfloor}\tt b$ and $(\tt ba)^{\lfloor \delta/2 \rfloor+1}\tt b$ (cf. Example~\ref{ex:factors}).

Let us fix an arbitrary letter $a$ and consider a factor in $\vv$ of minimal length containing two $a$'s: $v=aXa$. Let $u=\pi(v)$; statement 1 is thus equivalent to $|u|\ge 2\delta-1$. The structure of the sequences $\yy$ and $\yy'$ implies that $u$ contains $\delta+1$ occurrences of the letter $\pi(a)$. Since $\uu$ has no factor $\tt aa$, $\pi(a)={\tt a}$ implies $|u|\ge 2\delta+1$. Now let $\pi(a)=\tt b$. Note that the number of $\tt b$'s in two consecutive blocks is at least $2\lfloor \delta/2 \rfloor+2\ge \delta +1$. Since $|u|_{\tt b}=\delta+1$, the factor $u$ intersects at most three blocks in $\uu$. In a factor of a block, the number of $\tt b$'s exceeds the number of $\tt a$'s by at most 1. Hence, $|u|_{\tt b}-|u|_{\tt a}\le 3$, yielding $|u|\ge 2\delta-1$. Statement 1 is proved.

Now let $v=abXab$ be a shortest factor in $\vv$ with the property $\pi(ab)=\tt bb$. Let $u=\pi(v)$. The structure of $\yy$ and $\yy'$ implies $|u|_{\tt b}=c\delta+2$ for some $c\in \mathbb{N}$. The location of the factors $\tt bb$ in $u$ implies that $u=\mathtt{b}Z\mathtt{b}$, where $Z$ is a concatenation of blocks. Statement 2 is equivalent to $|Z|> 4\delta-4$. We have $|Z|_{\tt b}=c\delta$, while each block contains either $\lfloor \delta/2 \rfloor +1$ or $\lfloor \delta/2 \rfloor +2$ $\tt b$'s. Since $\delta\ge 6$, the number of $\tt b$'s in one block is 
less than $\delta$ while the number of $\tt b$'s in two blocks is greater than $\delta$. Hence $c\ge 2$. If $c=2$, then $Z$ consists of exactly three blocks, implying $|Z|_{\tt a}=|Z|_{\tt b}-3=2\delta-3$ and thus $|Z|=4\delta-3$. If $c>2$, $|Z|>4\delta-4$ holds trivially. Statement 2 is proved. 
\end{proof}

Let us recall that the length  $|v|$ of a return word equals the difference between two consecutive occurrences of $w$ in $\xx_{2\delta}$. Furthermore, if an index $i$ is an  occurrence of  $w$ in $\xx_{2\delta}$, then $i$ is an occurrence of $\pi(w)$ in $\uu$. Hence,  $\pi(v)$  is a concatenation of return words to $\pi(w)$ in $\uu$. Thus if $r$ and $s$ are the return words to $\pi(w)$ in $\uu$, then 
\begin{equation}\label{kl}
   \text{there exist} \  k,\ell \in \mathbb{N}, k+\ell\geq 1\ \ \text{such that} \ \ |v| = |\pi(v)| = k|r| + \ell|s|\,. 
   \end{equation}
In the sequel, if $\pi(w)$ is bispecial, we always assume that $r$ is the prefix return word to $\pi(w)$ and $s$ is the non-prefix one.

Of course,  not all pairs $(k,\ell)$  correspond to a concatenation of $r$ and $s$ forming  $\pi(v)$. The possible combinations are given by factors of the derived sequence of $\uu$ to  $\pi(w)$. Formally, 
\begin{equation}\label{klParikh}
   \begin{pmatrix} \ell \\ k \end{pmatrix} \ \text{ is the Parikh vector of a factor in \ } \  {\dd}_{\uu}(\pi(w)).
\end{equation}  

The simple form of $\yy$ and  $\yy'$ implies some evident properties.

\begin{lemma}\label{mod6} Let $v$ be a return word in $\xx_{2\delta}$ to a non-empty factor $w$ such that $\pi(w)$ contains both $\tt a$ and $\tt b$.
\begin{enumerate}
\item Both $|\pi(v)|_{\tt a}$ and $|\pi(v)|_{\tt b}$ are divisible by  $\delta$.
\item If $w$ is a bispecial factor of $\xx_{2\delta}$, then $\pi(w)$ is a bispecial factor of $\uu$.  
 \end{enumerate}
\end{lemma}
% \begin{lemma}\label{mod6} Let $v$ be a return word  in $\xx_{2\delta}$ to a non-empty factor $ w \in \mathcal{L}(\xx_{2\delta})$.
% \begin{enumerate}
% \item If $\tt a$ occurs in $ \pi(w)$, then $|\pi(v)|_{\tt a} = 0\mod \delta$.  
% \item If $\tt b$ occurs in $ \pi(w)$, then $|\pi(v)|_{\tt b} = 0\mod \delta$.
% \item If $\tt a$ and  $\tt b$ occur in $ \pi(w)$ and $w$ is a bispecial factor of $\xx_{2d}$, then $\pi(w)$ is a bispecial factor of $\uu$.  
%  \end{enumerate}
% \end{lemma}
Item 2 of Lemma~\ref{mod6} enables us to exploit Proposition \ref{prop:returnWords} on bispecial factors and return words in Sturmian sequences. 

\begin{lemma}\label{kl_mod_d} 
Let $w$ be a bispecial factor of $\xx_{2\delta}$ such that $\pi(w)$ contains both ${\tt a}$  and $\tt b$ and let $\theta'$ be the slope of the derived sequence of $\dd_{\uu}(\pi(w))$. 
Then for every return word $v$ to $w$ in $\xx_{2\delta}$,  there exist $k$ and $\ell$ such that 
\begin{enumerate}
    \item $k, \ell \in \mathbb{N}, k+\ell\geq 1$,
    \item  $k =0 \mod \delta,\ \  \ell = 0\mod \delta$, 
    \item  $\theta'(k-1) - 1 < \ell < \theta'(k+1) +1$ 
\end{enumerate}
and $|v| = k|r| + \ell|s|$, where $r$ is the prefix and $s$ the non-prefix return word to $\pi(w)$ in $\uu$. 
\end{lemma}
\begin{proof} 
We need  to show that $k, \ell$ in Equation \eqref{kl} have also the properties described in Items 2 and 3. 
Let us denote by  $A$  the  matrix from $ \mathbb{N}^{2\times 2}$ such that $ \vec{V}(s)$ and $ \vec{V}(r)$ are the first and the second column of $A$, respectively. Then $\vec{V}(\pi(v))= A\begin{pmatrix} \ell \\ k \end{pmatrix}$. By Item 1 of Lemma \ref{mod6}, $A\begin{pmatrix} \ell \\ k \end{pmatrix} =  \delta\begin{pmatrix} L \\ K \end{pmatrix}$ for some integers $L,K$. By   Proposition \ref{prop:returnWords}, $\det A=\pm 1$, and thus the inverse matrix $A^{-1}$ belongs to  $\mathbb{Z}^{2\times 2}$. Hence $\begin{pmatrix} \ell \\ k \end{pmatrix} =  \delta A^{-1}\begin{pmatrix} L \\ K\end{pmatrix}$ and Item 2 follows. 

Item 3 is a direct consequence of \eqref{klParikh} and  Lemma \ref{lem_kl}. 
\end{proof}

\subsection{Balanced sequence over $d$ letters with $d\geq 14$, $d$ even}

\begin{theorem} \label{t:d14}
Let $\uu$ be the standard sequence with the slope $\theta=[0,1,\lfloor \delta/2 \rfloor,\overline{1}]$, $\delta\ge 7$ be an integer, $\yy=(12\cdots \delta)^\omega$, $\yy'=(1'2'\cdots \delta')^\omega$. 
Then the balanced sequence $\xx_{2\delta}=\barva(\uu, \yy, \yy')$ over the alphabet of $2\delta$ letters has the critical exponent $\frac{2\delta-1}{2\delta-2}$.
\end{theorem}

% We have
% $$
% \uu = \underbrace{\tt ba\cdots ba}_{\lfloor \delta/2 \rfloor\times}{\tt b}\underbrace{\tt ba\cdots ba}_{(\lfloor \delta/2 \rfloor{+}1)\times}{\tt b}\underbrace{\tt ba\dots ba}_{\lfloor \delta/2 \rfloor\times}\tt b\cdots
% $$
The next proposition is crucial in proving Theorem~\ref{t:d14}.

\begin{proposition}\label{long_2d}
 Let $w$ be a bispecial factor of $\xx_{2\delta}$ such that $\pi(w)$ contains ${\tt a}$  and $\tt b$. Then  
$\frac{|w|}{|v|} < \frac1{2\delta-2}$ for every 
 return word $v$ to $w$ in $\xx_{2\delta}$. 
 \end{proposition}

\begin{proof} 
By  Lemma \ref{mod6}, $z:=\pi(w)$ is a bispecial factor in $\uu$. As $z$ contains both $\tt a$ and $\tt b$, it can be any  bispecial factor of $\uu$ except $\tt b$.  We use Proposition~\ref{prop:returnWords} to find the pair $(N,m)$  with $m< a_{N+1}$ corresponding to $z$ and the slope $\theta'$ of the derived sequence ${\dd}_{\uu}(z)$. By the definition of convergents, the pair $(1,0)$ corresponds to the factor $\tt b$. So we need to analyse all pairs $(N, 0)$ for $N\geq 2$ and all pairs $(1,m)$ for $m\in \{1,\dots, \lfloor \delta/2 \rfloor-1\}$. Recall that $\varphi=\frac{\sqrt{5}-1}{2}$ is the slope of the Fibonacci sequence.

\begin{itemize}
    \item The pair  $(N, 0)$, $N\geq 2$: \quad  By Proposition \ref{prop:returnWords}, $\theta' = [0,\overline{1}] =\varphi$. 
    
    Let $k,\ell$ satisfy Items 1--3 from Lemma~\ref{kl_mod_d}. Then $k=\delta k'$ and $\ell = \delta \ell'$ for some $\ell', k' \in \mathbb{N}, k'+\ell'\geq 1$. Dividing all parts of Item~3 by $\delta$, we get
    \begin{equation}\label{eq:kl_prime}
    \varphi k' - \frac{\varphi+1}{\delta} < \ell' < \varphi k' + \frac{\varphi+1}{\delta}\,.         
    \end{equation}
    As $\delta\geq 7$, the left inequality fails for $\ell'=0$, $k'>0$ and for $\ell'=1$, $k'>1$; the right inequality does not hold for $\ell'=k'=1$. Hence $\ell'\geq 2$ and $k'\geq 3$.
% $$ 0,618 k' -0,232 < \frac{\sqrt{5}-1}{2}k' - \frac{\sqrt{5}+1}{2\delta} < \ell' < \frac{\sqrt{5}-1}{2}k' + \frac{\sqrt{5}+1}{2\delta} < 0,619 k' + 0,232. 
% $$
    Then $|v|= \delta k'|r| + \delta \ell'|s|\geq 3\delta |r| + 2\delta |s|$.   By Proposition~\ref{prop:returnWords},  $|w| = |\pi(w)| = |r| + |s|-2$.  Thus 
    \begin{equation}\label{eq:w_to_v}
    \frac{|w|}{|v|}\leq  \frac{|r|+|s| - 2}{3\delta |r| +2\delta |s|}  < \frac1{2\delta } < \frac1{2\delta -2}\,.
    \end{equation}

    \item The pair  $(1, m)$ with $m\in \{1,\dots, \lfloor \delta/2 \rfloor-1\}$: \quad Using \eqref{eq:convergents} and Proposition~\ref{prop:returnWords}, we get  $|r|= p_1+q_1=2$, 
    $|s| = m(p_1+q_1)+p_0+q_0 = 2m+1$ and  $\theta' = [0,\lfloor \delta/2 \rfloor-m,\overline{1}]=\frac{1}{\lfloor \delta/2 \rfloor-m+\varphi}\le \frac{1}{1+\varphi}=\varphi$.  
 
    Let $k,\ell$ satisfy Items 1--3 from Lemma~\ref{kl_mod_d}. Then $k=\delta k'$ and $\ell = \delta \ell'$ for some $\ell', k' \in \mathbb{N}, k'+\ell'\geq 1$. Similar to \eqref{eq:kl_prime}, we get
    \begin{equation}\label{eq:kl_prime2}
    \theta' k' - \frac{\theta'+1}{\delta} < \ell' < \theta' k' + \frac{\theta'+1}{\delta}\,. 
    \end{equation}
    As $\theta'>\frac{2}{\delta}$, the lefthand part is positive for $k'\ne 0$, implying $\ell'>0$. As $\theta'\le \varphi$, the right inequality holds only if $\ell'<k'$. Thus if $\ell'\ge 2$, then we bound the ratio $\frac{|w|}{|v|}$ as in \eqref{eq:w_to_v}. It remains to study the case $\ell'=1$. Substituting the values of $\ell'$ and $\theta'$ into \eqref{eq:kl_prime2}, we get the following double inequality for $k'$:
    \begin{equation}\label{eq:k_prime}
    \Big\lfloor \frac{\delta}{2} \Big\rfloor - m + \varphi-\frac{\lfloor \frac{\delta}{2} \rfloor-m+1+\varphi}{\delta} < k' < \Big\lfloor \frac{\delta}{2} \Big\rfloor - m + \varphi+\frac{\lfloor \frac{\delta}{2} \rfloor-m+1+\varphi}{\delta}\,.
    \end{equation}
    Since $\delta\ge 7$ and $1\le m< \lfloor \delta/2 \rfloor$, we have $\varphi> \frac{\lfloor \delta/2 \rfloor-m+1+\varphi}{\delta}$. Hence the only integer solution of \eqref{eq:k_prime} is $k'=\lfloor \delta/2 \rfloor-m+1$. The fact that it is a solution also restricts $m$ through the inequality $\varphi+\frac{\lfloor \delta/2 \rfloor-m+1+\varphi}{\delta}>1$, which transforms to 
    \begin{equation}\label{eq:valueM}   m<\varphi(\delta+1)+\lfloor \delta/2 \rfloor+1-\delta\,.
    \end{equation}    
    % or, after dividing all parts by $\theta'$, 
    % $$ 
    % k'- \frac{\lfloor \delta/2 \rfloor-m+1+\varphi}{\delta} < \ell'(\lfloor \delta/2 \rfloor-m+\varphi) < k' + \frac{\lfloor \delta/2 \rfloor-m+1+\varphi}{\delta} 
    % $$

% There is no solution for $\ell'=0$. 

% If  $\ell'\geq 2$, then $k'\geq 1$ and we have
%  $$\frac{|w|}{|v|}\leq  \frac{|r|+|s| - 2}{\delta|r| + 2\delta|s|}  = \frac{2m+1}{2\delta + 2\delta(2m+1)}  < \frac{1}{2\delta} < \frac{1}{2\delta-2}.$$

% $$ \lfloor \delta/2 \rfloor-m+\underbrace{\gamma-\frac{\lfloor \delta/2 \rfloor+1-m+\gamma}{\delta}}_{A:=}<k' < \lfloor \delta/2 \rfloor-m+\underbrace{\gamma+\frac{\lfloor \delta/2 \rfloor+1-m+\gamma}{\delta}}_{B:=}\,.$$
% Using the facts that $1\leq m< \lfloor \delta/2 \rfloor$ and $\delta\geq 7$, we have  

% $A>0$ and $B<1.21$. An integer $k'$  satisfying the previous inequalities exists only if $B>1$.  More precisely,  $\ell' =1$ implies $k'= \lfloor \delta/2 \rfloor-m+1$ and 

    Estimating the ratio $|w|/|v|$, we get
    \begin{equation}\label{eq:w_to_v2}
    \frac{|w|}{|v|}\leq  \frac{|r|+|s| - 2}{k'\delta |r| + \delta |s|}  = \frac{2+(2m+1)-2}{(\lfloor \delta/2 \rfloor-m+1)\cdot\delta\cdot 2 + \delta\cdot(2m+1)} = \frac{2m+1}{\delta(2\lfloor \delta/2 \rfloor+3)}\,.
    \end{equation}
    % $$
    % |v| = k'\delta |r| + \delta |s| = (\lfloor \delta/2 \rfloor-m+1)\cdot\delta\cdot 2 + \delta\cdot(2m+1) = \delta(2\lfloor \delta/2 \rfloor+3).   
    % $$
%Length  of the return $v$ word  to $w$  corresponding to the pair  $k', \ell'$  is   

%\left\{ \begin{array}{ll}\delta(\delta +3), & \text{\quad if $\delta$ is even;}\\  \delta(\delta +2), & \text{\quad if $\delta$ is odd. } \end{array}\right.  
To finish the proof, it remains  to verify that  
\begin{equation} \label{eq:last}
\frac{2m+1}{\delta(2\lfloor \delta/2 \rfloor+3)} <\frac{1}{2\delta -2}\,.
\end{equation}
From \eqref{eq:valueM} we have $m=1$ for $\delta=7$ and $m\le 2$ for $\delta \in\{8,9,10\}$. So in all these cases \eqref{eq:last} trivially holds. If $\delta\geq 11$, the right hand side of \eqref{eq:valueM} satisfies $$
\varphi(\delta+1)+\lfloor \delta/2 \rfloor+1-\delta \leq  (\varphi - \tfrac12)\delta +1 + \varphi<\tfrac14 \delta + \tfrac14\,, 
$$ 
yielding 
$$ 
\frac{2m+1}{\delta(2\lfloor \delta/2 \rfloor+3)} < \frac{2\frac{\delta+1}{4} +1}{\delta(2\lfloor \delta/2 \rfloor+3)} \leq \frac{\delta+3}{2\delta(\delta+2)}< \frac{1}{2\delta -2}\,.
$$
\end{itemize}
\end{proof}

\begin{proof}[Proof of Theorem~\ref{t:d14}]
By Theorem~\ref{Prop_FormulaForCR}, it is sufficient to show that $\frac{|w|}{|v|}\le \frac{1}{2\delta-2}$ for every bispecial factor $w$ of $\xx_{2\delta}$ and every return word $v$ to $w$ in $\xx_{2\delta}$. If $\pi(w)$ contains both $\tt a$ and $\tt b$, the required inequality follows from Proposition~\ref{long_2d}. If $\pi(w)$ contains only one of these letters, the inequality follows from Proposition~\ref{p:short}. 
\end{proof}

\begin{remark} 
The inequalities in Proposition~\ref{long_2d} and in Item 2 of  Proposition~\ref{p:short} are strict, so the only type of factor of exponent $\frac{2\delta-1}{2\delta-2}$ in $\xx_{2\delta}$ is the repeat of a single letter at distance $2\delta-2$. As Proposition~\ref{p:maxfreq} shows, such repeats are unavoidable in balanced $d$-ary sequences with the critical exponent $\frac{d-1}{d-2}$.
\end{remark}

\begin{remark} 
The only step in the proof of Proposition~\ref{long_2d} that fails for $\delta=6$ is the derivation from inequality \eqref{eq:kl_prime}: the case $\ell'=1$, $k'=2$ becomes possible. Indeed, for $\delta = 6$, the slope  is $\theta = [0,1,3,\overline{1}]$.   By Proposition  \ref{prop:returnWords},  the Parikh vectors of the  return words $r$  and $s$ to the bispecial factor  corresponding to the pair $(7,0)$ are     $\vec{V}(r) = \begin{pmatrix} p_7 \\ q_7 \end{pmatrix}  = \begin{pmatrix} 29 \\ 37 \end{pmatrix} $   and $\vec{V}(s) = \begin{pmatrix} p_6 \\ q_6 \end{pmatrix}  = \begin{pmatrix} 18 \\ 23 \end{pmatrix} $.  
Then  
$$
\frac{|w|}{|v|} = \frac{|r|+|s| - 2}{k'\delta|r| + \ell'\delta |s|} = \frac{66+41 -2 }{2\cdot 6\cdot 66 + 1\cdot6\cdot 41} = \frac{105}{1038} > \frac{1}{10}\,. 
$$
In fact, $1038$ is the minimal period of a factor with the exponent $> \frac{11}{10}$.    
\end{remark}

\subsection{Balanced sequence over 12 letters}

\begin{theorem} \label{t:d12}
Let $\uu$ be the standard sequence with the slope $\theta=[0,1,3,\overline{2}]$,  $\yy=(123456)^\omega$, $\yy'=(1'2'3'4'5'6')^\omega$. Then the critical exponent of the balanced sequence $\xx_{12}=\barva(\uu, \yy, \yy')$ is  $\frac{11}{10}$.
\end{theorem}

% Thus
% $$
% \uu = {\tt bababab^2ababab^2abababab^2ababab^2ababa}\cdots
% $$

We prove the following analog of Proposition~\ref{long_2d}.

\begin{proposition}\label{long}
 Let $w$ be a bispecial factor of $\xx_{12}$ such that $\pi(w)$ contains ${\tt a}$  and $\tt b$. Then  
$\frac{|w|}{|v|} < \frac1{10}$ for every 
 return word $v$ to $w$ in $\xx_{12}$. 
 \end{proposition}

\begin{proof} 
By  Lemma \ref{mod6}, $z:=\pi(w)$ is a bispecial factor in $\uu$. As $z$ contains both $\tt a$ and $\tt b$, it can be any  bispecial factor of $\uu$ except $\tt b$.  We use Proposition~\ref{prop:returnWords} to find the pair $(N,m)$  with $m< a_{N+1}$ corresponding to $z$ and the slope $\theta'$ of the derived sequence ${\dd}_{\uu}(z)$. As the pair $(1,0)$ corresponds to $z=\tt b$, we need to analyse all pairs $(N, 0)$ and $(N,1)$ for $N\geq 2$ and the pairs $(1,1)$ and $(1,2)$.

\begin{itemize}
%     \item The pair  $(N, 0)$, $N\geq 2$: \quad  We apply Proposition \ref{prop:returnWords} for  $m=0$.  As $|r| =p_N+q_N$ and $|s| =p_{N-1}+q_{N-1}$, we have $|r|  = a_N|s| + p_{N-2}+q_{N-2}>  a_N |s| = 2|s|$.  $\theta' = [0,\overline{2}] =\sqrt{2}-1 $. 
    
%     Let us find $k,\ell$ satisfying all items in  Lemma 4. Obviously, $k=6k'$ and $\ell = 6 \ell'$ for some $\ell', k' \in \mathbb{N}, k'+\ell'\geq 1$. The pair $k', \ell'$   has to satisfy
% $$ (\sqrt{2}-1)k' - \frac{\sqrt{2}}{6} < \ell' < (\sqrt{2} - 1)k' + \frac{\sqrt{2}}{6}. 
% $$

    \item The pair  $(N, 0)$, $N\geq 2$: \quad  By Proposition \ref{prop:returnWords}, $\theta' =         [0,\overline{2}] =\sqrt{2}-1$. 
    
    Let $k,\ell$ satisfy Items 1--3 from Lemma~\ref{kl_mod_d}. Then $k=6k'$ and $\ell = 6\ell'$ for some $\ell', k' \in \mathbb{N}, k'+\ell'\geq 1$. Dividing all parts of condition~3 by $6$, we get
    \begin{equation}\label{eq:kl_prime6}
    (\sqrt{2}-1)k' - \frac{\sqrt{2}}{6} < \ell' < (\sqrt{2} - 1)k' + \frac{\sqrt{2}}{6}\,.         
    \end{equation}
    We successively obtain $\ell'\ge 1$ from the left inequality and $k'\ge2$ from the right inequality of \eqref{eq:kl_prime6}. By Proposition~\ref{prop:returnWords},  $|w| = |\pi(w)| = |r| + |s|-2$, $|r| =p_N+q_N$, and $|s| =p_{N-1}+q_{N-1}$.  As \eqref{eq:convergents} implies $|r|  = a_N|s| + p_{N-2}+q_{N-2}>  a_N |s| = 2|s|$, we have
    \begin{equation*}\label{eq:w_to_v6}
    \frac{|w|}{|v|}=  \frac{|r|+|s| - 2}{6k'|r| +6\ell'|s|}\leq  \frac{|r|+|s| - 2}{12|r| +6|s|} \leq  \frac{|r|+|s| - 2}{10|r| +10|s|}< \frac1{10}\,.
    \end{equation*} 

    \item The pair  $(N, 1)$, $N\geq 2$:\quad  By  Proposition \ref{prop:returnWords}, $\theta' = [0,1, \overline{2}] = \frac{1}{\sqrt{2}}$. Similarly to \eqref{eq:kl_prime6}, the pair $k', \ell'$ satisfies 
        $$ 
        \frac{1}{\sqrt{2}} k' -\frac{1+\sqrt{2}}{6\sqrt{2}} < \ell' < \frac{1}{\sqrt{2}} k' + \frac{1+\sqrt{2}}{6\sqrt{2}}\,.  
        $$
    From the left inequality, $\ell'\ne 0$ and, moreover, $\ell'=1$ implies $k'\le 1$, contradicting the right inequality. Hence $\ell'\geq 2$; now $k'\geq 3$ from the right inequality. Therefore,
    \begin{equation*}
    \frac{|w|}{|v|}=  \frac{|r|+|s| - 2}{6k'|r| +6\ell'|s|}\leq  \frac{|r|+|s| - 2}{18|r| +12|s|} < \frac1{10}.
    \end{equation*}     

    \item The pair  $(1, 1)$: \quad From \eqref{eq:convergents} and Proposition~\ref{prop:returnWords} we have  $|r|= p_1+q_1=2$, $|s| = p_1+q_1+p_0+q_0 = 3$ and  $\theta' = [0,\overline{2}] = \sqrt{2}-1$.  Since $\theta'$ is the same as in the case $(N,0), N\geq 2$, we have the same estimate
    $$
    \frac{|w|}{|v|}\leq  \frac{|r|+|s| - 2}{12|r| +6|s|}  = \frac1{14} < \frac{1}{10}\,.
    $$
 
    \item The pair  $(1, 2)$: \quad By \eqref{eq:convergents} and Proposition~\ref{prop:returnWords}, $|r|= p_1+q_1=2$, $|s| = 2(p_1+q_1)+p_0+q_0 = 5$ and  $\theta' = [0,1,\overline{2}] = \frac{1}{\sqrt{2}}$.  Since $\theta'$ is the same as in the case $(N,1), N\geq 2$, we have the same estimate
    $$
    \frac{|w|}{|v|}\leq  \frac{|r|+|s| - 2}{18|r| +12|s|}  = \frac5{96} <\frac{1}{10}\,.
    $$
 \end{itemize}
\end{proof}

\begin{proof}[Proof of Theorem~\ref{t:d12}]
By Theorem~\ref{Prop_FormulaForCR}, it is sufficient to show that $\frac{|w|}{|v|}\le \frac{1}{10}$ for every bispecial factor $w$ of $\xx_{12}$ and every return word $v$ to $w$ in $\xx_{12}$. If $\pi(w)$ contains both $\tt a$ and $\tt b$, the required inequality follows from Proposition~\ref{long}. If $\pi(w)$ contains only one of these letters, the inequality follows from Proposition~\ref{p:short}. 
\end{proof}

\section{Conclusion and Open Problems}

We have shown that for balanced sequences over a $d$-ary alphabet, the critical exponent is greater than or equal to $\frac{d-1}{d-2}$ for $d \geq 11$. In fact, the proved result is a bit stronger: a balanced $d$-ary sequence contains a factor of exponent at least $\frac{d-1}{d-2}$ (thus in the case $\mathrm{E}(\vv)=\frac{d-1}{d-2}$, the supremum is reached). Further, we have proved this lower bound sharp for all alphabets of even size $d\ge 12$ by presenting an explicit construction of balanced sequences with the required critical exponents. Based on these results, we state a new conjecture, replacing the conjecture from \citep{RSV19}, which fails for $d\ge 11$.

\begin{conjecture} \label{conj}
The minimal critical exponent of a $d$-ary balanced sequence with $d\geq 11$ equals $\frac{d-1}{d-2}$. 
\end{conjecture}

This conjecture remains open for alphabets of odd size. As the next step to set the conjecture, we have found, with the aid of computer search, a balanced sequence $\xx_{11}$ over an 11-letter alphabet with the required critical exponent $\frac{11}{10}$. The construction is very asymmetric and hard to be found by hand: $\xx_{11}=\barva(\uu, \yy, \yy')$,  where $\theta=[0,5,1,\overline{1,1,1,2}]$ and the two constant gap sequences are $\yy=(\tt 12)^{\omega}$  and   $\yy'=$
$$
\tt(1'2'3'4'5'6'7'8'1'9'3'2'5'4'7'6'1'8'3'9'5'2'7'4'1'6'3'8'5'9'7'2'1'4'3'6'5'8'7'9')^{\omega}\,.
$$
We have calculated the critical exponent of the sequence $\xx_{11}$ using our computer program based on the algorithm described in~\citep{DolceDP21}. As the proof is a tedious version of the proofs from Section \ref{s:sequences}, we have not included it in the paper.

\smallskip
Besides Conjecture \ref{conj}, we propose a few questions for further study of the minimal critical exponents of sequences given by certain natural restrictions.
\begin{itemize}
    \item A sequence $\uu\in\A^\mathbb{N}$ is \emph{$k$-balanced} if $\big||u|_a-|v|_a\big|\le k$ for any its factors $u$ and $v$ of equal length and any letter $a$. Thus 1-balanced sequences are exactly balanced sequences; the Thue-Morse sequence, having the minimal critical exponent among binary sequences, is 2-balanced.
    \begin{itemize}
        \item[\bf Q1] What is the minimal critical exponent of a $d$-ary 2-balanced sequence?
    \end{itemize}
    \item A sequence $\uu\in\A^\mathbb{N}$ is \emph{symmetric} if for any its factor $u$ and any bijection $\tau: \A\to\A$, $\uu$ has the factor $\tau(u)$. The Thue--Morse sequence, having the minimal critical exponent among binary sequences, and the Arshon sequence \citep{Ars37}, having the minimal critical exponent among ternary sequences, are symmetric.
    \begin{itemize}
        \item[\bf Q2] What is the minimal critical exponent of a $d$-ary symmetric sequence?
    \end{itemize}
    \item Two words are \emph{Abelian equivalent} if they have equal Parikh vectors. Replacing equality in the notion of a power with Abelian equivalence, Abelian powers, Abelian exponents and Abelian critical exponents are defined. There exist sequences with Abelian critical exponent arbitrarily close to 1 \citep{CaCu99}, but for no alphabet size $d$ the minimal Abelian critical exponent of a $d$-ary sequence is known; see \citep{SaSh12, PeSh21c} for the best known lower bounds.
    \begin{itemize}
        \item[\bf Q3] What is the minimal Abelian critical exponent of a $d$-ary balanced sequence?
    \end{itemize}
\end{itemize}

\section{Acknowledgement} 
The second author was supported  by Czech technical university in Prague, through the project  SGS20/183/OHK4/3T/14.
The first and the third authors were supported by the Ministry of Education, Youth and Sports of the Czech Republic through the project  CZ.02.1.01/0.0/0.0/16\_019/0000778. The fourth author acknowledges the support by the Ministry of Science and Higher Education of the Russian Federation (Ural Mathematical Center project No. 075-02-2021-1387).

\bibliography{biblio.bib}
\bibliographystyle{elsarticle-harv.bst}

\end{document}